\documentclass[a4paper,11pt]{article}
\usepackage[utf8]{inputenc}
\usepackage{amsfonts}
\usepackage{amsthm}
\usepackage{amsmath}
\usepackage{setspace}
\usepackage{array}
\usepackage{setspace}
\usepackage[margin=3cm]{geometry}
\numberwithin{equation}{section}
\usepackage{longtable}
\newtheorem{theorem}{Theorem}[section]
\newtheorem{lemma}[theorem]{Lemma}
\newtheorem{cor}[theorem]{Corollary}
\newtheorem{hypothesis}[theorem]{Hypothesis}
\theoremstyle{definition}

\theoremstyle{remark}
\newtheorem{remark}[theorem]{Remark}

\numberwithin{equation}{section}



\title{The Dirichlet Problem for Einstein Metrics on Cohomogeneity One Manifolds}
\author{Timothy Buttsworth}
\begin{document}
\maketitle
\begin{abstract}
Let $G/H$ be a compact homogeneous space, and let $\hat{g}_0$ and $\hat{g}_1$ be $G$-invariant Riemannian metrics on $G/H$. 
We consider the problem of finding a $G$-invariant Einstein metric $g$ on the manifold $G/H\times [0,1]$ subject to the 
constraint that $g$ restricted to $G/H\times \{0\}$ and $G/H\times \{1\}$ coincides with $\hat{g}_0$ and $\hat{g}_1$, respectively. 
By assuming that the isotropy representation of $G/H$ consists of pairwise inequivalent irreducible summands, 
we show that we can always find such an
Einstein metric.
\end{abstract}
\section{Introduction}
Let $M$ be a smooth manifold. This paper concerns finding Riemannian metrics $g$ on $M$ that are Einstein, i.e., whose Ricci curvature $Ric(g)$ satisfies 
\begin{equation}\label{EE}
Ric(g)=\lambda g
\end{equation}
for some constant $\lambda$ on $M$.
Physically, Einstein metrics are of interest because in the Lorentzian setting they describe the geometry of a vacuum according to Einstein's theory 
of relativity. 
They are also of fundamental interest in geometry because a manifold 
with an Einstein metric can be viewed as having `constant curvature'; see the introduction to the subject in Chapter 0 of \cite{Besse}. 

For open and closed manifolds, there are several results relating to the solvability of \eqref{EE}, 
and a large and detailed survey of some classical results appears in \cite{Besse}. Some more recent results are available in \cite{Aubin}. 
In addition to open and closed manifolds, it is natural to consider 
the Einstein equation, as well as other geometric PDEs, on manifolds with boundary, in which case one prescribes various boundary conditions. 
Anderson studies the problem of solving \eqref{EE} on manifolds with boundary in \cite{Anderson08}, 
and considers Dirichlet 
conditions, Neumann conditions, and prescribing the conformal class of the metric as well as the mean curvature 
at the boundary. He demonstrates that the Einstein equation with Dirichlet conditions is not 
Fredholm, which makes it difficult to study in general. However, this equation is Fredholm under the prescription of mean curvature and conformal class. 
The issues of choosing appropriate boundary conditions have also come up in other 
geometric equations. For example, in the study of the Ricci flow, Pulemotov and Gianniotis study boundary conditions involving the mean curvature and conformal class in
\cite{APQLRF} and \cite{PG} respectively, 
while Shen and Pulemotov study Robin-type and Neumann-type boundary conditions in \cite{Shen} and \cite{PulemotovRF} respectively. 
Pulemotov also prescribes Dirichlet conditions in his study of the prescribed Ricci curvature problem in \cite{APC}. 

Finding general results about the solvability of \eqref{EE} tends to be difficult, so an effort has been made to study the problem in simpler settings. 
For example, when our manifold $M$ 
is acted on transitively by some Lie group $G$, and we require
that our Einstein metric is 
invariant under the action of $G$, \eqref{EE} becomes a system of algebraic equations. Finding solutions to this system is known 
as the problem of finding homogeneous Einstein metrics 
and has been studied extensively in, for example, \cite{Park}, \cite{Wang}, \cite{BWZ},
\cite{Arvanitoyeorgos}, \cite{Jensen} and \cite{B}. 
After the homogeneous setting, the next natural step is requiring that our $d$-dimensional manifold $M$ has a 
$G$-action whose orbits are $(d-1)$-dimensional. Here, 
we say our manifold is cohomogeneity one, and it is natural to 
 restrict attention 
to metrics on $M$ that are $G$-invariant. In this case, \eqref{EE}, as well as many other geometric equations, becomes a system of ODEs rather than a system of PDEs, 
and results about existence seem easier to obtain. One of the first examples of a cohomogeneity one Einstein metric appeared in \cite{Page}, 
and subsequently, the general theory began to be developed by B\'erard-Bergery in \cite{BB}, and by Page and Pope in \cite{PP}. Since then, there has been much work 
done in the area of cohomogeneity one Einstein metrics. For example, in \cite{EschenburgWang}, Eschenburg and Wang 
study the initial value problem for the Einstein metric in a neighbourhood of a fixed orbit. 
The Einstein equation has also been studied by Dancer and Wang in \cite{DancerWang} by viewing the ODE as a Hamiltonian flow. The cohomogeneity one setting was 
also used by Pulemotov in his work on the prescribed Ricci 
curvature problem in \cite{APC}, in the study of Ricci solitons by Dancer and Wang in \cite{DancerWangS}, and in the work on the Ricci flow done by Pulemotov in
\cite{PulemotovRF} and by Bettiol and Krishnan in \cite{BetKri}. 

The setting of this paper is the study of the Einstein equation on cohomogeneity one manifolds $M$ subject to boundary conditions.  
We assume that $M$ appears as $G/H\times [0,1]$, where $G/H$ is a compact homogeneous space, and the boundary of our manifold $M$ is
$\left(G/H\times \{0\}\right)\cup \left(G/H\times \{1\}\right)$. The Dirichlet problem in this case consists 
in finding Einstein metrics that coincide 
with two fixed $G$-invariant Riemannian metrics $\hat{g}_0$ and $\hat{g}_1$ on $G/H\times \{0\}$ and $ G/H\times \{1\}$, respectively. 
We demonstrate in this paper that we can always find such an Einstein metric after we impose the ``monotypic'' assumption on the compact homogeneous space $G/H$. 
This assumption essentially allows us to diagonalise the Einstein equation, substantially simplifying analysis. 
This assumption has been used by a number of authors, for example in 
\cite{DancerWang} and \cite{GroveZiller}. 
\section{Preliminaries and the Main Result}
Before we state the main result, we will provide some background and notation, and state some assumptions. 
We let $G$ be a compact Lie group, 
and let $H$ be a closed Lie subgroup of $G$. We let $\mathfrak{g}$ and $\mathfrak{h}$ be the Lie algebras of $G$ and $H$ respectively. Once 
we choose some Ad$(G)$-invariant inner product $Q$ on $\mathfrak{g}$, 
we let 
$\mathfrak{m}$ be the $Q$-orthogonal complement 
of $\mathfrak{h}$ in $\mathfrak{g}$, and naturally identify $\mathfrak{m}$ with the tangent space of $G/H$ at $H$. Now take a $Q$-orthogonal 
decomposition  
\begin{equation}\label{DC}
\mathfrak{m}=\bigoplus_{i=1}^{n} \mathfrak{m}_i
\end{equation}
such that each $\mathfrak{m}_i$ in \eqref{DC} is an irreducible Ad$(H)$ submodule.
We make the following assumption on this decomposition. 
\begin{hypothesis}\label{IID}
The submodule $\mathfrak{m}_{i_1}$ is non-isomorphic to $\mathfrak{m}_{i_2}$ whenever $i_1\neq i_2$. 
\end{hypothesis}
\noindent
This assumption is of great convenience as it ensures that the decomposition \eqref{DC} is unique up to the order of summands. 
Furthermore, according to \cite[Lemma 1.1]{Park}, 
this assumption ensures that any diagonal metric respecting the decomposition 
\eqref{DC} has diagonal Ricci curvature which also respects this decomposition. 
We will also assume that the dimension of $\mathfrak{m}$ is strictly greater than $1$. 

The cohomogeneity one manifold we will study is $G/H\times [0,1]$. 
The boundary of this manifold is $(G/H\times \{0\})\cup (G/H\times \{1\})$, so prescribing Dirichlet conditions involves fixing $\hat{g}_0$ and $\hat{g}_1$, 
two $G$-invariant Riemannian metrics on $G/H$. 
Hypothesis \ref{IID} implies the existence of two arrays of positive numbers $(a_i)_{i=1}^{n}$ and $(b_i)_{i=1}^{n}$ such that 
\begin{align*}
\hat{g}_0(X,Y)&=\sum_{i=1}^{n} a_i^2 Q(pr_{\mathfrak{m}_i} X,pr_{\mathfrak{m}_i}Y)\\
\hat{g}_1(X,Y)&=\sum_{i=1}^{n} b_i^2 Q(pr_{\mathfrak{m}_i} X,pr_{\mathfrak{m}_i}Y)
\end{align*}
for all $X,Y\in \mathfrak{m}$. Here, 
$pr_{\mathfrak{m}_i} X$ denotes the $Q$-orthogonal projection of the vector $X$ onto $\mathfrak{m}_i$.  

We can now state our main result. 
\begin{theorem}\label{MRCHEM}
Let $G/H$ be a compact homogeneous space of dimension at least $2$, and suppose that Hypothesis \ref{IID} holds. 
Fix any two $G$-invariant Riemannian metrics $\hat{g}_0$ and $\hat{g}_1$ on $G/H$. Then 
there exists a $G$-invariant Einstein metric $g$ on $G/H\times [0,1]$ such that the metrics on $G/H\times \{0\}$ and 
$G/H\times \{1\}$ induced by $g$ coincide with $\hat{g}_0$ and $\hat{g}_1$, respectively. 
\end{theorem}
To prove this result, we look for $G$-invariant Riemannian metrics on $G/H\times [0,1]$ having the form 
\begin{equation}\label{RMG}
g=h^2(r)dr\otimes dr+g_r,
\end{equation}
where $r$ is the natural parameter running through the interval $[0,1]$, 
$h(r)$ is a smooth positive function on $[0,1]$, and 
$g_r$ is a one-parameter collection of $G$-invariant Riemannian metrics on $G/H$ satisfying 
\begin{align*}
g_r(X,Y)&=\sum_{i=1}^{n} f_i^2(r) Q(pr_{\mathfrak{m}_i} X,pr_{\mathfrak{m}_i}Y)
\end{align*}
for all $X,Y\in \mathfrak{m}$, where $(f_i)_{i=1}^{n}$ is some collection of smooth positive functions on the interval $[0,1]$. Now, Lemma 3.1 of \cite{APC} states that the 
Ricci curvature of such Riemannian metrics is given by 
$Ric(g)=H(r) dr\otimes dr+R_r$, where 
\begin{align*}
H(r)=-\sum_{k=1}^{n}d_k\left(\frac{f_k''}{f_k}-\frac{h'f_k'}{hf_k}\right)
\end{align*}
and
\begin{align*}
&R_r(X,Y)=\\
&\sum_{i=1}^{n}\left(\frac{\beta_i}{2}+\sum_{k,l=1}^{n}\gamma^{l}_{ik}\frac{f_i^4-2f_k^4}{4 f_k^2f_l^2}
-\frac{f_if_i'}{h}\sum_{k=1}^{n}d_k\frac{f_k'}{hf_k}+\frac{f_i'^2}{h^2}-\frac{f_if_i''}{h^2}+\frac{f_i h'f_i}{h^3}\right)Q(pr_{\mathfrak{m}_i} X,pr_{\mathfrak{m}_i}Y).
\end{align*} 
Here, $\beta_i$ and $\gamma_{ik}^{l}$ are constants associated with the choice of scalar product $Q$ and the homogeneous space $G/H$, and
$d_i$ is the dimension of the submodule $\mathfrak{m}_i$. See \cite{GroveZiller} for the precise definitions of these numbers. We also let $d=\sum_{i=1}^{n}d_i$, and we recall that $d>1$. 
The Einstein equation is diffeomorphism invariant, 
so we can assume that $h$ is constant on $[0,1]$. For the moment, we will assume that $h=1$, in which 
case the Riemannian metric $g$ satisfies \eqref{EE} if and only if 
\begin{equation}\label{EM1}
\sum_{i=1}^{n}d_i\left(
\frac{\beta_i}{2f_i^2}+\sum_{k,l=1}^{n}\gamma^{l}_{ik}\frac{f_i^4-2f_k^4}{4 f_i^2f_k^2f_l^2}-\frac{f_i'}{f_i}\sum_{k=1}^{n}d_k\frac{f_k'}{f_k}+\frac{f_i'^2}{f_i^2}\right)=(d-1)\lambda,
\end{equation}
\begin{equation}\label{EM2}
\frac{\beta_i}{2f_i^2}+\sum_{k,l=1}^{n}\gamma^{l}_{ik}\frac{f_i^4-2f_k^4}{4 f_i^2f_k^2f_l^2}-\frac{f_i'}{f_i}\sum_{k=1}^{n}d_k\frac{f_k'}{f_k}+\frac{f_i'^2}{f_i^2}-\frac{f_i''}{f_i}=\lambda
\end{equation}
for $i=1,\cdots, n$. 
The rest of this paper is devoted to solving equations \eqref{EM1} and \eqref{EM2} subject to the boundary conditions $f_i(0)=a_i$ and $f_i(1)=b_i$. 
Before we do anything else, we state a result which simplifies this task, and is a consequence of the second contracted Bianchi identity.  
Since we impose Hypothesis \ref{IID}, the result follows from Lemma 2.4 of \cite{EschenburgWang}, but we prove it again for completeness. 
\begin{lemma}\label{PIE}
Let $(X_i)_{i=1}^{d}$ be a $Q$-orthonormal basis of $\mathfrak{m}$ adapted to our decomposition \eqref{DC} and let 
$e_{d+1}$ be the vector field $\partial r$. 
Suppose $I$ is a subinterval of $\mathbb{R}$ containing $0$ and $g$ is a $G$-invariant Riemannian metric on $G/H\times I$ having the form \eqref{RMG} with $h=1$. 
Assume that there exists a constant $\lambda$ such that for $i=1,\cdots, d$, 
$Ric(X_i,X_i)=\lambda g(X_i,X_i)$ at $(H,r)\in G/H\times I$ for all $r\in I$. If $Ric(e_{d+1},e_{d+1})=\lambda$ 
at $(H,0)$, then $Ric(e_{d+1},e_{d+1})=\lambda$ at $(H,r)$ for all $r\in I$. 
\end{lemma}
\begin{proof}
Let $e_i=\frac{X_i}{f_j(r)}$ at each $r\in I$, where $j$ is chosen so that $\mathfrak{m}_j$ contains $X_i$. 
Then $(e_i)_{i=1}^{d+1}$ is an orthonormal basis for $g$ at $(H,r)\in G/H\times [0,1]$ for each $r\in [0,1]$, 
and we have $Ric(g)(e_i,e_i)=\lambda g(e_i,e_i)=\lambda$ unless $i=d+1$. 
We extend 
this orthonormal basis to a local orthonormal basis of vector fields so that $(e_i)_{i=1}^{d}$ is a collection of vector fields on $G/H$. 
Using the second contracted Bianchi identity, as well as the fact that 
$Ric(g)(e_i,e_i)$ is constant for $i=1,\cdots, d$, we then find that along $(H,r)\in G/H\times I$ 
\begin{align*}
\frac{1}{2}e_{d+1}Ric(g)(e_{d+1},e_{d+1})&=\frac{1}{2}e_{d+1}\left(\sum_{i=1}^{d+1}Ric(g)(e_i,e_i)\right)\\
&=\frac{1}{2}e_{d+1}(S(g))\\
&=\sum_{i=1}^{d+1}\nabla_{e_i}Ric(g)(e_i,e_{d+1})\\
&=e_{d+1}(Ric(g)(e_{d+1},e_{d+1}))\\
&-\left(\sum_{i=1}^{d+1}Ric(g)(\nabla_{e_i}e_i,e_{d+1})+Ric(g)(\nabla_{e_i}e_{d+1},e_i)\right),
\end{align*}
where $S(g)$ is the scalar curvature of $g$. 
Rearranging and using the Koszul formula, we see that 
\begin{align*}
\frac{1}{2}e_{d+1}Ric(g)(e_{d+1},e_{d+1})&=\sum_{i=1}^{d+1}g(\nabla_{e_i}e_i,e_{d+1})Ric(g)(e_{d+1},e_{d+1})\\
&+\sum_{i=1}^{d+1}g(\nabla_{e_i}e_{d+1},e_i)Ric(g)(e_i,e_i)\\
&=\sum_{i=1}^{d}\left(g(\nabla_{e_i}e_i,e_{d+1})Ric(g)(e_{d+1},e_{d+1})+\lambda g(\nabla_{e_i}e_{d+1},e_i)\right)\\
&=\sum_{i=1}^{d}\left(-g(\nabla_{e_i}e_{d+1},e_i)Ric(g)(e_{d+1},e_{d+1})+\lambda g(\nabla_{e_i}e_{d+1},e_i)\right)\\
&=\left(\sum_{i=1}^{d}g(\nabla_{e_i}e_{d+1},e_i)\right)(-Ric(g)(e_{d+1},e_{d+1})+\lambda).
\end{align*}
Letting $y(r)=Ric(g)(e_{d+1},e_{d+1})$, we see that 
\begin{align*}
\frac{1}{2}y'(r)=\left(\sum_{i=1}^{d}g(\nabla_{e_i}e_{d+1},e_i)\right)(\lambda-y(r)).
\end{align*}
Since $Ric(g)(e_{d+1},e_{d+1})=\lambda$ at $(H,0)$, we know that $y(0)=\lambda$. Basic ODE theory then implies that $y(r)=Ric(g)(e_{d+1},e_{d+1})=\lambda$ for all $r\in I$. 
\end{proof}
Lemma \ref{PIE} implies that to find solutions to \eqref{EM1} and \eqref{EM2}, it suffices to find a solution of 
\begin{equation}\label{EM3}
\sum_{i=1}^{n}d_i\left(
\frac{\beta_i}{2f_i^2}+\sum_{k,l=1}^{n}\gamma^{l}_{ik}\frac{f_i^4-2f_k^4}{4 f_i^2f_k^2f_l^2}-\frac{f_i'}{f_i}\sum_{k=1}^{n}d_k\frac{f_k'}{f_k}
+\frac{f_i'^2}{f_i^2}\right)\biggr\rvert_{r=0}=(d-1)\lambda
\end{equation}
and
\begin{equation}\label{EM4}
\frac{\beta_i}{2f_i^2}+\sum_{k,l=1}^{n}\gamma^{l}_{ik}\frac{f_i^4-2f_k^4}{4 f_i^2f_k^2f_l^2}-\frac{f_i'}{f_i}\sum_{k=1}^{n}d_k\frac{f_k'}{f_k}+\frac{f_i'^2}{f_i^2}-\frac{f_i''}{f_i}=\lambda
\end{equation}
for $i=1,\cdots, n$.
We will study the Einstein equation in this form because \eqref{EM3} is an equation in $\mathbb{R}$ for the real parameter $\lambda$ and \eqref{EM4} is an equation in 
$C^2([0,1];\mathbb{R}^n)$ for the vector function $f\in  C^2([0,1];\mathbb{R}^n)$. 
\section{Torus}
In this section, we consider the situation that our homogeneous space $G/H$ is the $d$-dimensional torus $\mathbb{T}^d$. In this case, $\mathfrak{m}$ is completely reducible into one-dimensional modules. 
Therefore, Hypothesis~\ref{IID} is violated, but solutions to 
\eqref{EM1} and \eqref{EM2} still define Einstein metrics as long as 
$\beta_i=0$ for $i=1,\cdots,d$, $\gamma_{ik}^{l}=0$ for all $i,k,l=1,\cdots,d$ and $d_i=1$ for $i=1,\cdots, d$. 
This is evident from the discussion in Section 1 of \cite{GroveZiller}, for example. 
The main result of this section is the following, and it will
help us prove Theorem \ref{MRCHEM}. 
\begin{theorem}\label{MRT}
 Let $d_i=1$, $\gamma_{ik}^{l}=\beta_i=0$. There is a unique pair $(\lambda,f)=(\lambda,f_1,\cdots,f_d)
 \in \mathbb{R}\times C^2([0,1]; (\mathbb{R}^{+})^d)$ solving \eqref{EM3} and \eqref{EM4} such that 
 $f_i(0)=a_i$ and $f_i(1)=b_i$ for each $i=1,\cdots,d$. 
\end{theorem}

To prove this theorem, we introduce the diagonal matrix $L$ with diagonal entries $L_i=\frac{f_i'}{f_i}$. 
We see that solving \eqref{EM3} and \eqref{EM4} is equivalent to solving 
\begin{equation}\label{TL}
(d-1)\lambda=\left(tr(L^2)-tr(L)^2\right)\vert_{r=0}
\end{equation}
and 
\begin{equation}\label{TM}
L'=-tr(L)L-\lambda.
\end{equation}
Once we solve \eqref{TL} and \eqref{TM} for $L$, we see from the definition of 
$L$ that $f_i(r)=s_ie^{\int_{0}^{r}L_i}$ for some positive scaling constants $s_i$. 
By evaluating $f_i(r)$ at $r=0$ and $r=1$ and taking the ratio of the two, 
we can see that these constants $s_i$ can be chosen so that $f_i(0)=a_i$ and $f_i(1)=b_i$ if and only if 
\begin{align}\label{ICT}
\int_{0}^{1}L_i=D_i, 
\end{align}
 where $D_i=\ln(b_i)-\ln(a_i)$ and we will define $D=\sum_{i=1}^{d}D_i$. 
 This shows us that proving Theorem~\ref{MRT} is equivalent to uniquely solving \eqref{TL}, \eqref{TM} and \eqref{ICT} for any constants $D_i\in \mathbb{R}$. 
To do this, we first solve \eqref{TM} and \eqref{ICT} in terms of $\lambda$ and then show that $\lambda$ can be uniquely chosen so that \eqref{TL} holds. 
  
\begin{lemma}\label{SWL}
 Let $\lambda$, $D_i$ be constants and let $D=\sum_{i=1}^{d}D_i$. There exists a solution $L$ of \eqref{TM} and \eqref{ICT} if and only if 
 $\lambda<\frac{\pi^2}{d}$. The solution is unique.  
\end{lemma}
\begin{proof} 
Notice that whenever \eqref{TM} and \eqref{ICT} hold, 
we must have 
\begin{equation}\label{TLL}
tr(L')=-tr(L)^2-d\lambda
\end{equation}
and 
\begin{equation}\label{TILT}
\int_{0}^{1}tr(L)=D.
\end{equation}
We consider separately the cases that $\lambda$ is negative, positive and $0$. In each case, we show that unless $\lambda\ge \frac{\pi^2}{d}$, 
we can uniquely solve \eqref{TLL} and \eqref{TILT}. 
We then use the expression for $tr(L)$ to uniquely solve \eqref{TM} and \eqref{ICT}. In the case that $\lambda$ is positive, we show that 
we need $\lambda<\frac{\pi^2}{d}$ for a solution to exist. This will complete the proof.
 
\vspace{0.5cm}

\noindent \textbf{First} \textbf{Case}. We assume $\lambda<0$, and set $\mu=-\lambda>0$. The general solution of \eqref{TLL} is given by 
\begin{equation}\label{NLNC}
tr(L)=\sqrt{d\mu}\frac{e^{2\sqrt{d\mu}r}-C}{e^{2\sqrt{d\mu}r}+C} 
\end{equation}
 for some 
constant $C$, or $tr(L)=-\sqrt{d\mu}$.
We see that $tr(L)=-\sqrt{d\mu}$ satisfies \eqref{TILT} if and only if $D=-\sqrt{d\mu}$. Otherwise,
$tr(L)$ must be given by \eqref{NLNC}, and we note that for this expression to be defined and continuous in $r$ on $[0,1]$, we need $C\notin [-e^{2\sqrt{d\mu}},-1]$. 
Furthermore,
\eqref{TILT} holds if and only if the constant $C$ is chosen so that 
\begin{align*}
D&=\int_{0}^{1}tr(L)\\
&=\left(\ln(|C+e^{2\sqrt{d\mu}r}|)-\sqrt{d\mu}r\right)\vert_{r=0}^{r=1}\\
&=\ln(|C+e^{2\sqrt{d\mu}}|)-\sqrt{d\mu}-\ln(\left|C+1\right|).
\end{align*}
This is equivalent to 
\begin{align}\label{MCE}
 \left|\frac{C+e^{2\sqrt{d\mu}}}{C+1}\right|&=e^{D+\sqrt{d\mu}}.
\end{align}
Since $C\notin [-e^{2\sqrt{d\mu}},-1]$, we notice that 
$\left|\frac{C+e^{2\sqrt{d\mu}}}{C+1}\right|=\frac{C+e^{2\sqrt{d\mu}}}{C+1}$, and since $D+\sqrt{n\mu}\neq 0$, we can uniquely solve \eqref{MCE} for $C$ with 
\begin{align*}
 C&=\frac{e^{2\sqrt{d\mu}}- e^{D+\sqrt{d\mu}}}{ e^{D+\sqrt{d\mu}}-1},
\end{align*}
and we note that $C\notin [-e^{2\sqrt{d\mu}},-1]$ does indeed hold. 
Now that we have solved \eqref{TLL} and \eqref{TILT}, we will solve \eqref{TM} and \eqref{ICT}. 
From \eqref{TM} we find that
\begin{align*}
 L_i&=\frac{\mu(e^{2\sqrt{d\mu r}}-C)}{\sqrt{d\mu}(C+e^{2\sqrt{d\mu}r})}
 +\frac{c_i\sqrt{d\mu}e^{\sqrt{d\mu}r}}{C+e^{2\sqrt{d\mu}r}}\\
 &=\frac{1}{d}tr(L)+\frac{c_i\sqrt{d\mu}e^{\sqrt{d\mu}r}}{C+e^{2\sqrt{d\mu}r}}
\end{align*}
for some constants $c_i$, provided we are not at the stationary solution $tr(L)=-\sqrt{d\mu}$. 
The constants $c_i$ can be found from \eqref{ICT} after integrating, and we see 
\begin{equation}\label{CI}
 D_i-\frac{D}{d}=c_i\int_{0}^{1}\frac{\sqrt{d\mu}e^{\sqrt{d\mu}r}}{C+e^{2\sqrt{d\mu}r}}.
\end{equation}
On the other hand, if we are at the stationary solution $tr(L)=-\sqrt{d\mu}$, then the solution of \eqref{TM} 
is given by $L_i=\frac{1}{d}tr(L)+c_i\sqrt{d\mu}e^{\sqrt{d\mu}r}$ and the constants $c_i$ are chosen so that 
\begin{equation}\label{CICP}
 D_i-\frac{D}{d}=c_i\int_{0}^{1}\sqrt{d\mu}e^{\sqrt{d\mu}r}.
\end{equation}

\vspace{0.5cm}

\noindent \textbf{Second} \textbf{Case}. Now we assume $\lambda>0$, in which case \eqref{TLL} implies that 
\begin{equation}\label{PLTE}
tr(L)=-\sqrt{d\lambda}\tan(C+\sqrt{d\lambda}r),
\end{equation}
for some constant $C$. For $tr(L)$ to be defined on $[0,1]$, we require that $\cos(C+\sqrt{d\lambda}r)$ does not change sign. This tells us that we need $\sqrt{d\lambda}<\pi$, which we assume is the case from now on. We can also add $\pi$ to $C$ if necessary to ensure that $\cos(C+\sqrt{d\lambda}r)$ is positive for all $r\in [0,1]$, because such a change in $C$ does not change the value of $tr(L)$. 
 
Equation \eqref{TILT} then implies that 
\begin{align*}
D=\ln \cos(C+\sqrt{d\lambda})-\ln \cos(C),
\end{align*}
 so we can see that $C$ is chosen such that 
\begin{align*}
e^{D}&=\frac{\cos(\sqrt{d\lambda}+C)}{\cos(C)}\\
&=\cos(\sqrt{d\lambda})-\tan(C)\sin(\sqrt{d\lambda}).
\end{align*}
Rearranging, we see that 
\begin{align*}
C&=\arctan\left(\frac{\cos(\sqrt{d\lambda})-e^{D}}{\sin(\sqrt{d\lambda})}\right)\in [-\frac{\pi}{2},\frac{\pi}{2}],
\end{align*}
and it is straightforward to show that given this choice of $C$, $\cos(C+\sqrt{d\lambda}r)$ is indeed positive for $r\in [0,1]$. 
Using \eqref{PLTE}, we can now solve equation \eqref{TM} to find that 
\begin{align*}
L_i&=c_i\sec(C+\sqrt{d\lambda}r)+\frac{1}{d}tr(L)
\end{align*}
for some constants $c_i$.
Equation \eqref{ICT} then implies that the $c_i$ constants must be chosen so that
\begin{align*}
D_i-\frac{D}{d}&=\frac{c_i}{\sqrt{d\lambda}}\ln(\tan(C+\sqrt{d\lambda}r)+\sec(C+\sqrt{d\lambda}r))\vert_{r=0}^{r=1},
\end{align*}
and we note that the input of the logarithm is indeed positive. 

\vspace{0.5cm}

\noindent \textbf{Third} \textbf{Case}. Finally, if $\lambda=0$, we split into two further cases: $D\neq 0$ and $D=0$. If $D\neq 0$, \eqref{TLL} and \eqref{TILT} implies that 
$tr(L)=\frac{1}{C+r}$ for some constant $C\notin[-1,0]$ chosen so that $\ln\left(\frac{C+1}{C}\right)=D$. Rearranging gives
$\frac{1}{e^D-1}=C$ which makes sense as $D\neq 0$. 
Then \eqref{TM} implies that $L_i=\frac{tr(L)}{d}+\frac{c_i}{C+r}$ and \eqref{ICT} implies that the $c_i$ terms are uniquely chosen so that $D_i-\frac{D}{d}=c_i D$. 

If $D=0$, then \eqref{TLL} combined with \eqref{TILT} implies that 
$tr(L)=0$. Equations \eqref{TM} and \eqref{ICT} then imply that $L_i=D_i$. 
\end{proof}

For a given $\lambda<\frac{\pi^2}{d}$, Lemma \ref{SWL} implies that a solution of \eqref{TM} and \eqref{ICT} exists and is unique. 
As found in the proof of Lemma \ref{SWL}, the solution is 
\begin{align}\label{DOL}
\begin{split}
L_i(r)=
 \begin{cases}
  \vspace{0.2cm}
  \frac{\mu(e^{2\sqrt{d\mu }r}-C_-)}{\sqrt{d\mu}(C_-+e^{2\sqrt{d\mu}r})}
 +\frac{c_{-i}\sqrt{d\mu}e^{\sqrt{d\mu}r}}{C_-+e^{2\sqrt{d\mu}r}}&  \text{if} -\lambda=\mu>0 \ \text{and} \ \sqrt{d\mu}\neq -D,\\ \vspace{0.2cm}
 \frac{-\sqrt{d\mu}}{d}+\tilde{c}_{-i}\sqrt{d\mu}e^{\sqrt{d\mu}r} &  \text{if} -\lambda=\mu>0 \ \text{and} \ \sqrt{d\mu}= -D,\\ \vspace{0.2cm}
 \frac{1}{d(C_0 +r)}+\frac{c_{0i}}{C_0+r}&  \text{if} \ \lambda=0 \ \text{and} \ D\neq 0, \\ \vspace{0.2cm}
 D_i& \text{if} \ \lambda=0 \ \text{and} \ D= 0, \\ \vspace{0.2cm}
 c_{+i}\sec(C_{+}+\sqrt{d\lambda}r)-\frac{\sqrt{d\lambda}}{d}\tan({C_+} +\sqrt{d\lambda}r)& \text{if} \ 0<\lambda<\frac{\pi^2}{d},
 \end{cases}
 \end{split}
\end{align}
where 
\begin{align*}
 C_-=\frac{e^{2\sqrt{d\mu}}- e^{(D+\sqrt{d\mu})}}{ e^{(D+\sqrt{d\mu})}-1},\  
  c_{-i}&=\frac{D_i-\frac{D}{d}}{\int_{0}^{1}\frac{\sqrt{d\mu}e^{\sqrt{d\mu}r}}{C_-+e^{2\sqrt{d\mu}r}}},\ 
 \tilde{c}_{-i}=\frac{D_i-\frac{D}{d}}{\int_{0}^{1}\sqrt{d\mu}e^{\sqrt{d\mu}r}},\\
 C_0=\frac{1}{e^D-1},\ & 
 c_{0i}=\frac{D_i-\frac{D}{d}}{D},\\
 C_+=\arctan\left(\frac{\cos(\sqrt{d\lambda})-e^{D}}{\sin(\sqrt{d\lambda})}\right),\ &  
 c_{+i}=\frac{\sqrt{d\lambda}(D_i-\frac{D}{d})}{\ln(\tan(C_++\sqrt{d\lambda}r)+\sec(C_++\sqrt{d\lambda}r))\vert_{r=0}^{r=1}}.
\end{align*}

We will put these solutions of \eqref{TM} and \eqref{ICT} into \eqref{TL} to find the value of $\lambda$. 
The following lemma imposes some initial constraints on the possible values of $\lambda$.

\begin{lemma}\label{LMBI}
 If the unique solution of \eqref{TM} and \eqref{ICT} satisfies
 \eqref{TL}, then $\lambda\in [\frac{-D^2}{d},\frac{\pi^2}{d})$.
\end{lemma}
\begin{proof}
We already know that $\lambda<\frac{\pi^2}{d}$ for a solution of \eqref{TM} to even exist. Now, if $\lambda=-\mu < -\frac{D^2}{d}$, then 
 $C_->0$. 
Substituting our solution of \eqref{TM} and \eqref{ICT} into \eqref{TL} yields 
$(1-d)\mu=d\mu\frac{(1-C_-)^2}{(1+C_-)^2}(\frac{1}{d}-1)+\frac{\sum_{i=1}^{d}d\mu c_{-i}^2}{(C_-+1)^2}$ which can be rearranged to 
\begin{equation}\label{TL0}
 4(\frac{1}{d}-1)C_-=\sum_{i=1}^{d}c_{-i}^2. 
\end{equation}
This is a contradiction since $C_->0$. 
\end{proof}
Now finding a value for $\lambda\in [-\frac{D^2}{d},\frac{\pi^2}{d})$ such that \eqref{TL} 
is satisfied will involve the function $m:[-\frac{D^2}{d},\frac{\pi^2}{d})\to \mathbb{R}^{+}$ given by 
\begin{align*}
\vspace{0.2cm}m(\lambda)&=
 \begin{cases}
 0& \text{if} \ 0<d\mu=-d\lambda =D^2,\\ \vspace{0.2cm}
  \left|\ln\left(\frac{(\frac{e^{\sqrt{d\mu}}}{\sqrt{-C_-}}+1)(-\frac{1}{\sqrt{-C_-}}+1)}{(1-\frac{e^{\sqrt{d\mu}}}
  {\sqrt{-C_-}})(1+\frac{1}{\sqrt{-C_-}})}\right)\right|&\text{if} \ 0<d\mu=-d\lambda< D^2,\\ \vspace{0.2cm}
   \left|D\right|&\text{if} \ \lambda=0,\\ \vspace{0.2cm}
  \left|\ln\frac{\tan(C_++\sqrt{d\lambda})+\sec(C_++\sqrt{d\lambda})}{\tan(C_+)+\sec(C_+)}\right|&\text{if} \ 0<d\lambda<\pi^2,
 \end{cases}
\end{align*}
where we treat $C_-$ and $C_+$ as functions of $\lambda$. 
The importance of $m$ is demonstrated with the following lemma. 
\begin{lemma}\label{LFL}
The solution of \eqref{TM} and \eqref{ICT} satisfies \eqref{TL} if and only if 
\begin{equation}\label{BEFF}
m(\lambda)=\sqrt{\frac{d}{d-1}\sum_{i=1}^{d}(D_i-\frac{D}{d})^2}. 
\end{equation}
\end{lemma}
\begin{proof}
First, if $-\frac{D^2}{d}=\lambda<0$, then $\sqrt{d\mu}=\pm D$, and $L_i(r)$ is given by the second line of \eqref{DOL}, or by the first line with $C_-=0$. 
Substituting these expressions into \eqref{TL} 
gives $\sum_{i=1}^{d}\tilde{c}_{-i}^2=0$ or $\sum_{i=1}^{d}c_{-i}^2=0$. This is equivalent to $\sum_{i=1}^{d}(D_i-\frac{D}{d})^2=0$, which is equivalent to 
$m(\lambda)=0$ as required. 

Next, assume that $-\frac{D^2}{d}< \lambda<0$. Then $C_-<0$, so like in the proof of Lemma \ref{LMBI} we deduce that \eqref{TL} is equivalent to 
\begin{equation}\label{TL01}
 4\left(\frac{1}{d}-1\right)C_-=\sum_{i=1}^{d}c_{-i}^2.
\end{equation}
By substituting our definitions of $C_-$ and $c_{-i}$ into 
\eqref{TL01}, explicitly 
evaluating the integral in the definition of $c_{-i}$ 
and rearranging for $\sum_{i=1}^{d}(D_i-\frac{D}{d})^2$, we find that \eqref{TL01} is equivalent to 
\eqref{BEFF}.  

If $\lambda=0$ and $D=0$, then substituting our solution of \eqref{TM} and \eqref{ICT} into \eqref{TL} gives $\sum_{i=1}^{d}(D_i-\frac{D}{d})^2=0$ as required. 
If $D\neq 0$, and $\lambda=0$, then the same process gives 
$\frac{1}{C_0^2}(\frac{1}{d}-1+\sum_{i=1}^{d}c_{0i}^2)=0$. Since $C_0\neq 0$, 
we can multiply this equation by $C_0^2$ and use our expression for $c_{0i}$ to find that \eqref{TL} is equivalent to \eqref{BEFF}. 

If $\lambda>0$, then \eqref{TL} is equivalent to 
\begin{equation}\label{LPRF}
(d-1)\lambda=(1-d)\lambda\tan^2(C_+)+\sec^2(C_+)\sum_{i=1}^{d}c_{+i}^2. 
\end{equation}
After using our definitions of $C_+$ and $c_{i+}$ and rearranging we see that \eqref{LPRF} is equivalent to \eqref{BEFF}. 
\end{proof}
The combination of Lemmas \ref{SWL}, \ref{LMBI} and \ref{LFL} demonstrates that uniquely solving \eqref{TL}, \eqref{TM} and \eqref{ICT} is 
equivalent to uniquely solving the equation $m(\lambda)=\sqrt{\frac{d}{d-1}\sum_{i=1}^{d}(D_i-\frac{D}{d})^2}$. 
The following two lemmas show us that this is indeed possible, and conclude the proof of Theorem \ref{MRT}. 
\begin{lemma}\label{FIS}
 The function $m:[\frac{-D^2}{d},\frac{\pi^2}{d})\to \mathbb{R}^+$ is surjective.
\end{lemma}
\begin{proof}
We demonstrate that the image of $m$ is $\mathbb{R}^{+}$ by computing some limits. 
First we check limits of $m$ on $(-\frac{D^2}{d},0)$, assuming $D\neq 0$. 
As $\mu\to 0$, the quantity $\frac{1-\frac{1}{\sqrt{-C_-}}}{1-\frac{e^{\sqrt{d\mu}}}{\sqrt{-C_-}}}$
converges to $e^{-D}$ by L'H\^opital's rule. Therefore, 
\begin{align*}
 \lim_{\lambda\to 0^{-}}m(\lambda)=\left|D\right|.
\end{align*}
As $\sqrt{d\mu}\to \left|D\right|$, $C_-$ goes to $0$ or $- \infty$ (depending on the sign of $D\neq 0$), so 
\begin{align*}
\lim_{\lambda\to -\frac{D^2}{d}}m(\lambda)=0.
\end{align*}
Now for the limits on $(0,\frac{\pi^2}{d})$. 
We see that 
$$\frac{\tan(C_++\sqrt{d\lambda})+\sec(C_++\sqrt{d\lambda})}{\tan(C_+)+\sec(C_+)}=\frac{\cos(C_+)}{\cos(C_++\sqrt{d\lambda})}\frac{\sin(C_++\sqrt{d\lambda})+1}{\sin(C_+)+1}$$
and 
$\frac{\cos(C_+)}{\cos(C_++\sqrt{d\mu})}=e^{-D}$ by the identity $\frac{\cos(x+y)}{\cos(x)}=\cos(y)-\tan(x)\sin(y)$. 
If $D\le 0$, then $C_+$ stays away from $-\frac{\pi}{2}$ as $\lambda\to 0^+$, so $\lim_{\lambda\to 0^+}\frac{\sin(C_++\sqrt{d\lambda})+1}{\sin(C_+)+1}=1$. On the other hand, if $D>0$, then $C_+\to -\frac{\pi}{2}$ as $\lambda\to 0^+$, so 
by L'H\^opitals rule, 
\begin{align*}
\lim_{\lambda\to 0^+}\frac{\sin(C_++\sqrt{d\lambda})+1}{\sin(C_+)+1}=
\lim_{\lambda\to 0^+}\frac{\cos(C_++\sqrt{d\lambda})}{\cos(C_+)}\left(1+\frac{\sin^2(0)+(\cos(0)-e^{D})^2}{e^{D}-1}\right)=e^{2D}.
\end{align*}
In either case, we can see that 
\begin{align*}
\lim_{\lambda\to 0^+}m(\lambda)=\left|D\right|. 
\end{align*}
It is clear that $m$ is continuous on $(-\frac{D^2}{d},0)$ and $(0,\frac{\pi^2}{d})$. These three limits demonstrate that $m$ is also continuous at $0$ and $-\frac{D^2}{d}$, 
so $m$ is in fact continuous 
on all of $[-\frac{D^2}{d},\frac{\pi^2}{d})$.

To conclude the proof, notice that if $\sqrt{d\lambda}\to \pi$, 
then $C_+\to -\frac{\pi}{2}$, so $\frac{\sin(C_++\sqrt{d\lambda})+1}{\sin(C_+)+1}$ goes 
to $\infty$, and 
\begin{align*}
\lim_{\lambda\to \frac{\pi^2}{d}}m(\lambda)=\infty.  
\end{align*}
Therefore, $m$ is continuous, gets arbitrarily close to $0$ and also gets arbitrarily large.  
The intermediate value theorem then implies that $m$ achieves all values in $\mathbb{R}^{+}$. 
\end{proof}
\begin{lemma}
 The function $m$ is monotone increasing.
\end{lemma}
\begin{proof}
The proof involves computing $m'(\lambda)$ for any $\lambda\in (-\frac{D^2}{d},\frac{\pi^2}{d})\setminus\{0\}$, and demonstrating that it is positive. 
The computation is straightforward, but is also tedious, so it is omitted. 
\end{proof}
Now that we have existence and uniqueness of solutions of \eqref{TL}, \eqref{TM} and \eqref{ICT}, we conclude this section by demonstrating that these solutions behave 
well under changes of values of $D_i$. 
\begin{lemma}\label{P3DC}
 Suppose that $\left|D_i\right|\le R$ for each $i=1,\cdots,d$. Then there exists $R'>0$ depending only on $R$ such that the solution $(\lambda,L_i)$ of 
 \eqref{TL}, \eqref{TM} and \eqref{ICT} satisfies $\left|\lambda\right|<R'$ and $\left|L_i\right|<R'$ for each $i=1,\cdots,d$. 
\end{lemma}
\begin{proof}
 Assume to the contrary that no such $R'>0$ exists. 
 Then there exists an unbounded sequence of solutions $(\lambda^j,L_i^j)$ to \eqref{TL}, \eqref{TM} and \eqref{ICT} with 
 $\left|D_i^j\right|\le R$ for each $i=1,\cdots,d$. Here, $j\in \mathbb{N}$ is used to distinguish the different elements of the sequence. 
 
 By taking a subsequence of $(\lambda^j,L_i^j)$ if necessary, we can assume that 
 the $\lambda^j$ terms are monotone increasing or decreasing. We already know that $\lambda^j\in [-\frac{(D^j)^2}{d},\frac{\pi^2}{d}]\subseteq [-d R^2,\frac{\pi^2}{d}]$, 
 so the $\lambda^j$ terms are bounded, 
 hence convergent. We claim that $\lambda_j$ converges to $\frac{\pi^2}{d}$ from below. If this were not the case, then 
 there would be some $K<\frac{\pi^2}{d}$ such that $-dR^2\le \lambda^j\le K$. Taking the trace of \eqref{TL} implies that 
 $tr(L^j)'=-tr(L^j)^2-d\lambda^j$ subject to $\int_{0}^{1}tr(L^j)=D^j$. Similarly to the proof of Lemma \ref{SWL}, we deduce that the solution of this equation is 
\begin{align*}
 \begin{split}
tr(L^j)(r)=
 \begin{cases}
  \frac{\sqrt{d\mu^j}(e^{2\sqrt{d\mu^j }r}-C^j_{-})}{(C^j_{-}+e^{2\sqrt{d\mu^j}r})}&  \text{if} -\lambda^j=\mu^j>0 \ \text{and} \ \sqrt{d\mu^j}\neq -D^j,\\
-\sqrt{d\mu^j}&  \text{if} -\lambda^j=\mu^j>0 \ \text{and} \ \sqrt{d\mu^j}= -D^j,\\
 \frac{1}{(C^j_{0} +r)}&  \text{if} \ \lambda^j=0 \ \text{and} \ D^j\neq 0, \\
 0& \text{if} \ \lambda^j=0 \ \text{and} \ D^j= 0, \\
 -\sqrt{d\lambda^j}\tan({C^j_{+}} +\sqrt{d\lambda^j}r)& \text{if} \ 0<\lambda^j<\frac{\pi^2}{d},\\
 \end{cases}
 \end{split} 
\end{align*}

where 
\begin{align*}
 C^j_{-}&=\frac{e^{2\sqrt{d\mu^j}}- e^{(D^j+\sqrt{d\mu^j})}}{ e^{(D^j+\sqrt{d\mu^j})}-1}, \\
 C^j_{0}&=\frac{1}{e^{D^j}-1},\\
 C^j_{+}&=\arctan\left(\frac{\cos(\sqrt{d\lambda^j})-e^{D^j}}{\sin(\sqrt{d\lambda^j})}\right).
\end{align*}
 Since
 $-dR^2\le \lambda^j\le K<\frac{\pi^2}{d}$ and $\left|D^j\right|\le dR$, it is straightforward to show that $tr(L^j)$ is bounded independently of $j$. 
 Now if we treat $tr(L^j)$ as a given function, we can think of \eqref{TL} coupled with \eqref{ICT} as a linear equation for $L_i^j$. This equation is easily solved 
 in terms of $tr(L^j)$, and the bound on $tr(L^j)$ then implies that $L_i^j$ itself is bounded independently of $j$ for each $i=1,\cdots,d$. However, 
 we now have bounds on both $\lambda^j$ and $L_i^j$, which contradicts the assumption that $(\lambda^j,L_i^j)$ is unbounded. 
 
 We now know that $\lambda^j\to \frac{\pi^2}{d}$ from below. However, in this case, 
 $C_+^j\to -\frac{\pi}{2}$, which implies that $m(\lambda^j)\to \infty$. 
 This is a contraction because $m(\lambda^j)$ must coincide with $\sqrt{\frac{d}{d-1}\sum_{i=1}^{d}(D_i^j-\frac{D^j}{d})^2}$, which is bounded. 
\end{proof}

\section{General Cohomogeneity One Manifolds}
In this section, we use Schauder degree theory to prove Theorem \ref{MRCHEM}. Before we do this, however, 
we briefly recall some relevant information about the Schauder degree. 
Let $X$ be a Banach space and let $I:X\to X$ be the identity mapping. 
Choose some $z\in X$, some bounded convex open subset $\Omega\subset X$ and some function $k:\bar{\Omega}\to X$, where $\bar{\Omega}$ means the closure of $\Omega$. 
The Schauder degree of $I-k$ in $\Omega$ over $z$ is denoted
$deg(I-k,\Omega,z)$, and can be defined whenever $k$ is a completely continuous map and $z\notin (I-k)(\partial \Omega)$. 
The following theorem is standard and states several facts about the Schauder degree that we will use in this section. These results and many others 
relating to the Schauder degree can be found in \cite{ASDT}, for instance. 
\begin{theorem}\label{SDTT} 
The Schauder degree has the following properties: 

(i) If $deg(I-k,\Omega,0)\neq 0$, there exists $x\in \Omega$ such that $k(x)=x$.

(ii) If $0\in \Omega$, then 
$deg(I,\Omega,0)=1$. 

(iii) If $H:[0,1]\times \bar{\Omega}\to X$ is a completely continuous map such such that $H(t,x)\neq x$ for all 
$x\in \partial \Omega$ and $t\in [0,1]$, then 
$deg(I-H(t,\cdot),\Omega,0)$ is independent of $t\in [0,1]$. 
\end{theorem}
Property $(iii)$ is referred to as the homotopy invariance of the Schauder degree. The strategy of this section is to use homotopies to deform the 
identity mapping into a mapping whose fixed points correspond to solutions of \eqref{EM1} and \eqref{EM2} with Dirichlet conditions, and then use Theorem~\ref{SDTT} 
to prove the existence of a solution. 
To do this, let $X=\mathbb{R}\times C^1([0,1];\mathbb{R}^n)$, let $\tilde{S}:\mathbb{R}^n\to \mathbb{R}$ and $\tilde{R}:\mathbb{R}^n\to \mathbb{R}^n$ 
be some continuous functions, let $c_0,c_1$ be vectors in $\mathbb{R}^n$, and define 
$H:[0,1]\times X\to X$ by
\begin{align*}
H&(t,\lambda,y)=\biggr(\frac{p_1(t)}{d-1}\sum_{i=1}^{n}d_i\left(
-y_i'\sum_{k=1}^{n}d_ky_k'+y_i'^2\right)\biggr\rvert_{r=0}+\frac{p_4(t)\tilde{S}(y)}{d-1}\biggr\rvert_{r=0},\\
&p_3(t)c_0(1-r)+p_3(t)c_1r+\int_{0}^{1}G(x,r)\bigg(-p_2(t)y'(x)\sum_{k=1}^{n}d_ky_k'(x)-p_2(t)\lambda +p_4(t)\tilde{R}(y(x))\bigg)dx\biggr).\\
\end{align*}
Here, $p_1,p_2,p_3,p_4:[0,1]\to [0,1]$ are defined by 
\begin{align*}
 p_i(t)=\begin{cases}
         0 & \ \text{if}\ t\le  \frac{i-1}{4},\\
         4t-(i-1) &\ \text{if}\ \frac{i-1}{4}< t\le  \frac{i}{4},\\
         1 &\ \text{if}\ t> \frac{i}{4},
        \end{cases}
\end{align*}
and $G:[0,1]\times [0,1]\to \mathbb{R}$ is the Green's function for the equation $y''=0$ on $[0,1]$ defined by 
\begin{align*}
 G(x,r)&=\begin{cases}
          x(r-1) & \text{if} \ 0\le x\le r\le 1,\\
          (x-1)r & \text{if} \ 0 \le r< x\le 1.
         \end{cases}
\end{align*}

Note that $\mathbb{R}\times C^1([0,1];\mathbb{R}^n)$ is a Banach space under the product norm, and $H$ is completely continuous by the Arzela-Ascoli Theorem. 
Furthermore, for a given $t\in [0,1]$, fixed points $(\lambda,y)$ of $H(t,\cdot)$ are exactly those solving 
\begin{align}\label{ETPG}
\begin{split}
p_4(t)\tilde{S}(y)\vert_{r=0}+p_1(t)\sum_{i=1}^{n}d_i\left(
-y_i'\sum_{k=1}^{n}d_ky_k'+y_i'^2\right)\biggr\rvert_{r=0}&=(d-1)\lambda,\\
p_4(t)\tilde{R}_i(y)-p_2(t)y_i'\sum_{k=1}^{n}d_ky_k'-y_i''&=p_2(t)\lambda, \ i=1,\cdots,n,\\
y(0)&=p_3(t)c_0,\\
y(1)&=p_3(t)c_1.
\end{split}
\end{align}
We will frequently use the fact that fixed points of $H(t,\cdot)$ are in one-to-one correspondance with solutions of \eqref{ETPG}. 
From now on, we set $c_0=(\ln(a_1),\cdots,\ln(a_n))$, $c_1=(\ln(b_1),\cdots,\ln(b_n))$ 
and 
\begin{align*}
\tilde{S}(y)&= \sum_{i=1}^{n}d_i\left(
\frac{\beta_i}{2e^{2y_i}}+\sum_{k,l=1}^{n}\gamma^{l}_{ik}\frac{e^{4y_i}-2e^{4y_k}}{4 e^{2y_i+2y_k+2y_l}}\right),\\
\tilde{R}_i(y)&=\frac{\beta_i}{2e^{2y_i}}+\sum_{k,l=1}^{n}\gamma^{l}_{ik}\frac{e^{4y_i}-2e^{4y_k}}{4 e^{2y_i+2y_k+2y_l}}.
\end{align*}
If we make the transformation $y_i=\ln(f_i)$ with these choices of $c_0,c_1,\tilde{S}$ and $\tilde{R}$, and set $t=1$, then  
the problem of finding a pair $(\lambda,y)$ solving \eqref{ETPG} is the same as solving \eqref{EM3} and \eqref{EM4} subject to the Dirichlet conditions. 

The main task of this section is to construct an appropriate bounded set $\Omega$, so that $deg(I-H(\frac{3}{4},\cdot),\Omega,0)\neq 0$. 
To find such a set, we need to construct a bounded set which contains the fixed points of $H(t,\cdot)$ for any $t\in [0,\frac{3}{4}]$. 
The next lemma describes the fixed points of $H(t,\cdot)$ for $t\in [0,\frac{1}{2}]$. 
\begin{lemma}\label{1100}
Fix $t\in [0,\frac{1}{2}]$. Then $(\lambda,y)=(0,0)$ is the unique solution of \eqref{ETPG}, hence the unique fixed point of $H(t,\cdot)$.
\end{lemma}
\begin{proof}
Suppose $(\lambda,y)$ is a solution of \eqref{ETPG}. 
By our choice of $t$, we know that $p_3(t)=p_4(t)=0$, and either $p_2(t)=0$ and $p_1(t)\in [0,1]$, or $p_1(t)=1$ and $p_2(t)\in (0,1]$. 
If $p_2=0$, then the second equation of \eqref{ETPG} coupled with the Dirichlet conditions implies that $y=0$. 
The first equation of \eqref{ETPG} then implies that $\lambda=0$ as well. 

Now suppose $p_1=1$ and $p_2\in (0,1]$. The numbers $(d_j)_{j=1}^{n}$ partition the set $\{1,\cdots,d\}$ into $n$ distinct subsets 
$$S_j=\left\{\sum_{k=1}^{j-1}d_k+1,\sum_{k=1}^{j-1}d_k+2,\cdots, \sum_{k=1}^{j}d_k\right\}$$ of cardinality $d_j$, and for each $i\in \{1,\cdots, d\}$, define 
$L_i=p_2y_j'$, where $j$ is chosen so that $i\in S_j$. Then 
$L_i$ solves \eqref{TL}, \eqref{TM} and \eqref{ICT} for $D_i=D=0$ and $\lambda$ replaced by $p_2^2(t)\lambda$. 
The uniqueness of solutions described in Section 3 implies that $L_i=p^2_2(t)\lambda=0$, so $y_i=0$ for each $i=1,\cdots,n$ and $\lambda=0$.  
\end{proof}
 The following lemma describes the fixed points of $H$ for $t\in [\frac{1}{2},\frac{3}{4}]$. 
\begin{lemma}\label{1110}
There exists a bounded convex open set $\Omega\subset \mathbb{R}\times C^1([0,1];\mathbb{R}^n)$ 
containing all the solutions $(\lambda,y)$ of \eqref{ETPG} for $t\in [\frac{1}{2},\frac{3}{4}]$, 
hence containing all the fixed points of $H(t,\cdot)$.  
\end{lemma}
\begin{proof}
By our choice of $t$, we know that $p_1(t)=p_2(t)=1$, $p_4(t)=0$, and $p_3(t)\in [0,1]$. 
As in the proof of Lemma \ref{1100} partition the set $\{1,\cdots,d\}$ into $n$ distinct subsets 
$S_j$, and for each $i\in \{1,\cdots, d\}$, define 
$L_i=y_j'$, where $j$ is chosen so that $i\in S_j$.
Then $(\lambda,L_i)$ solves \eqref{TL}, \eqref{TM} and \eqref{ICT}, except now, $D_i=p_3(\ln(b_j)-\ln(a_j))$. 
Therefore, the proof will be complete if we can 
demonstrate that any solution $(\lambda,L_i)$ of these three equations is bounded independently of $p_3\in [0,1]$. Such a bound is provided by Lemma \ref{P3DC}. 
\end{proof}
From now on, we fix some $\Omega$ satisfying the conclusion of Lemma \ref{1110}. 
\begin{cor}\label{T34}
On $\Omega$, we have  
$deg(I-H(\frac{3}{4},\cdot),\Omega,0)=1$. 
\end{cor}
\begin{proof}
If $t=\frac{1}{2}$, then the solution of \eqref{ETPG} is $(0,0)$ by Lemma \ref{1100}, so the set
$\Omega$ must contain $(0,0)$. 
Therefore, by Lemma \ref{1100}, no fixed points of $H(t,\cdot)$ occur on $\partial \Omega$ for any $t\in [0,\frac{1}{2}]$.  
Furthermore, Lemma \ref{1110} implies that no fixed points of $H(t,\cdot)$ 
occur on $\partial \Omega$ for any $t\in [\frac{1}{2},\frac{3}{4}]$. 
Therefore, properties $(ii)$ and $(iii)$ of Theorem \ref{SDTT} imply that $deg(I-H(\frac{3}{4},\cdot),\Omega,0)=deg(I-H(0,\cdot),\Omega,0)=1$. 
\end{proof}
We are now in a position to prove existence of solutions to \eqref{ETPG}. 
\begin{lemma}\label{TFL}
 Set $t=1$. There exists some $q\in (0,1]$ such that \eqref{ETPG} has a solution on $[0,\sqrt{q}]$ with the constraint $y_i(1)=\ln(b_i)$ replaced by 
 $y_i(\sqrt{q})=\ln(b_i)$. 
\end{lemma}
\begin{proof}
If we can show that $deg(I-H(1,\cdot),\Omega,0)=1$, then we have a fixed point of $H(1,\cdot)$. This fixed point is then a solution of \eqref{ETPG} with $t=1$ as required. Therefore, by Corollary \ref{T34} and Theorem \ref{SDTT}, 
it suffices to demonstrate that no fixed points of $H(t,\cdot)$ occur on $\partial \Omega$ for any $t\in [\frac{3}{4},1]$. 
By definition of $\Omega$, we know this cannot occur if $t=\frac{3}{4}$. However, if we do have a fixed point $(\lambda,y)$ on $\partial \Omega$ 
for some $\frac{3}{4}<t\le 1$, then define 
$\bar{y}_i:[0,\sqrt{p_4(t)}]\to \mathbb{R}$ by $\bar{y}_i(x)=y_i(\frac{x}{\sqrt{p_4(t)}})$ and $\bar{\lambda}=\frac{\lambda}{p_4(t)}$.
Then $(\bar{\lambda},\bar{y})$ is a solution of \eqref{ETPG} with the $1$ in $y_i(1)=b_i$ replaced by $\sqrt{p_4(t)}$. Taking $q=p_4(t)$ completes the proof. 
\end{proof}
Using the transformation $f_i=e^{y_i}$, Lemma \ref{TFL} gives us a solution $(\lambda,f)$ of 
\eqref{EM3} and \eqref{EM4} alongside the Dirichlet conditions $f_i(0)=a_i$ and $f_i(\sqrt{q})=b_i$. Therefore, this lemma demonstrates that for any two $G$-invariant Riemannian metrics $\hat{g}_0$ and $\hat{g}_1$ on $G/H$, there exists a 
$q>0$ such that there exists a $G$-invariant Einstein metric 
on $G/H\times [0,\sqrt{q}]$ that coincides with $\hat{g}_0$ and $\hat{g}_1$ on $G/H\times \{0\}$ and $G/H\times \{\sqrt{q}\}$, respectively. 
Since the Einstein equation is diffeomorphism invariant, the pullback of our solution under any $G$-invariant diffeomorphism 
from $G/H\times [0,1]$ to $G/H\times [0,\sqrt{q}]$ will also 
be Einstein with the boundary conditions preserved 
(the diffeomorphism sending $(x,r)\in G/H\times [0,1]$ to $(x,\sqrt{q}r)\in G/H\times [0,\sqrt{q}]$ will work). 
This proves Theorem \ref{MRCHEM}. 
\begin{remark}\label{EP}
Although we have an Einstein metric on $G/H\times [0,1]$, having to make $q$ smaller in Lemma \ref{TFL} corresponds to shortening the length of 
the interval $[0,1]$ with respect to the Einstein metric, which is equivalent to shrinking $h$, the constant defined in Section $2$. 
In Section $3$, no such shortening was required. In fact,
the Einstein equation on $\mathbb{T}^d\times \mathbb{R}$ is invariant under the transformation 
$\bar{f}(r)=f(qr)$ and $\bar{\lambda}=q^2\lambda$, so in addition to requiring that our Einstein metric satisfy the Dirichlet conditions, we can also prescribe its induced length of the interval $[0,1]$ arbitrarily. 
\end{remark}

\section{An Example of Non-Existence}
One might expect from our arguments of Section 3 and Remark \ref{EP} that when solving the Einstein equation, we can prescribe the length of the 
interval $[0,1]$ arbitrarily. 
In this section, we demonstrate that we cannot expect this to occur in general by exploring 
the Einstein equation on $G/H\times [0,1]$, where $G/H$ is the compact homogeneous space $\mathbb{S}^1\times \mathbb{S}^2$ being acted on transitively by 
$G=SO(2)\times SO(3)$. After choosing an $Ad(G)$-invariant metric $Q$ on the Lie algebra of $SO(2)\times SO(3)$ and following the arguments of Section 2, we see that 
the equations \eqref{EM1} and \eqref{EM2} become 
\begin{align}\label{ACE}
\begin{split}
 -\frac{f_1''}{f_1}-2\frac{f_2''}{f_2}&=\lambda,\\
 -\frac{f_1''}{f_1}-2\frac{f_1'f_2'}{f_1f_2}&=\lambda,\\
 -\frac{f_2''}{f_2}-\frac{(f_2')^2}{f_2^2}-\frac{f_1'f_2'}{f_1f_2}+\frac{\mu}{f_2^2}&=\lambda,
 \end{split}
\end{align}
where $\mu>0$ is some number depending on the choice of $Q$. We can eliminate $\lambda$ from these equations, and we find that 
\begin{align}\label{ACE1}
\begin{split}
 \frac{f_2''}{f_2}-\frac{f_1'f_2'}{f_1f_2}&=0,\\
 \frac{f_1''}{f_1}-\frac{f_2'^2}{f_2^2}+\frac{\mu}{f_2^2}&=0. 
 \end{split}
\end{align}
These equations also appear in Section 5 of \cite{BB}. 

Suppose we require that $f_2(0)=f_2(1)=\bar{f}_2$. Then there exists some $r$ such that $f_2'(r)=0$, and 
the first equation of \eqref{ACE1} implies that $f_2'(r)=0$ for all $r\in [0,1]$. Therefore, $f_2(r)=\bar{f}_2$ for all $r\in [0,1]$ and 
the second equation of \eqref{ACE1} becomes 
$\frac{f_1''}{f_1}+\frac{\mu}{\bar{f}_2^2}=0$. It is well-known that this equation is not solvable for small $\bar{f}_2$ 
if we require $f_1>0$ on $[0,1]$. This demonstrates that 
the Einstein equation with these Dirichlet conditions cannot be solved if we also require the length of the interval to be $1$. 
As a consequence, we see that we cannot expect to prescribe length arbitrarily, even though we can in the torus case.

\section{Acknowledgements}
I am grateful to Artem Pulemotov, J{\o}rgen Rasmussen, Joseph Grotowski and Andrew Dancer for helpful discussions on the problem. 
This research was supported by the Australian Government's Research Training Program scholarship and Artem Pulemotov's Discovery Early Career Researcher Award 
DE150101548.

\end{document}